\documentclass{amsart}
\usepackage{amsmath,amssymb,amsthm,amscd,amsfonts}
\usepackage{graphicx}%
\usepackage{enumerate}
\usepackage{tikz-cd}
\usepackage[colorlinks=true,urlcolor=red,citecolor=black,linkcolor=blue]{hyperref} 
\usepackage{color}
\usepackage[left=1.5in,right=1.5in,bottom=1.5in,top=1.5in]{geometry}
\usepackage{stackrel}
\usepackage{quiver}
\newenvironment{nouppercase}{%
  \renewcommand{\uppercasenonmath}[1]{}}{}

\theoremstyle{plain} \numberwithin{equation}{section}
\newtheorem{thm}{Theorem}[section]
\newtheorem{cor}[thm]{Corollary}

\newtheorem{lem}[thm]{Lemma}

\newtheorem*{clm*}{Claim}

\newtheorem{propp}{Proposition}

\newtheorem{thmm}[propp]{Theorem}
\newtheorem{corr}[propp]{Corollary}
\newtheorem{conjj}[propp]{Conjecture}

\theoremstyle{definition}
\newtheorem{defn}[thm]{Definition}

\newtheorem{rmk}[thm]{Remark}

\newcommand{\diam}{\operatorname{diam}}

\newcommand{\im}{\operatorname{im}}

\newcommand{\calS}{\mathcal S}
\newcommand{\bfs}{\mathbf c}
\newcommand{\bft}{\mathbf d}

\newcommand{\Z}{\mathbb Z}
\newcommand{\N}{\mathbb N}

\renewcommand{\phi}{\varphi}
\newcommand{\mesh}{\operatorname{mesh}}

\newcommand{\U}{\mathcal U}
\newcommand{\V}{\mathcal V}

\newcommand{\bfsi}{{}_ib_\bfs}

\begin{document}

\date{}
\title{A canonical splitting of the first homology group of Peano continua}
\author{Gregory R. Conner}
\author{Wolfgang Herfort}
\author{Curtis Kent}
\author{Petar Pave\v{s}i\'{c} }


\begin{abstract}
    The first singular homology of a Peano continuum $X$ with torsion-free first \v Cech homology, $\check H_1(x)$, splits as $H_1(X) = \check H_1(X) \oplus K$ where $K$ is the homology shape kernel of $X$.  Consequently if a Peano continuum $X$ is a subspace of $\mathbb R^3$, then $H_1(X) = \mathbb Z^\lambda \oplus K$ where $K$ is the homology shape kernel of $X$ and $\lambda$ is a countable cardinal. In the process we construct cotorsion quotients of subgroups of the first homology which correspond to path-connected fibrations of $X$.

\end{abstract}

\vspace*{-2cm}
\begin{nouppercase}
\maketitle
\end{nouppercase}

\setcounter{section}{1}
\section*{Introduction}

The first singular homology of spaces which are not locally contractible can be very complicated, see \cite{BarrattMilnor1962, cf, eda4, Eda2016}. In particular for a one-dimensional Peano continuum $X$, Eda showed that $H_1(X)$ is free abelian of finite rank or $H_1(X) = \mathbb Z^\mathbb N \oplus P/S$ where $P/S=\mathbb Z^{\mathbb N}/\bigoplus_{i\in\mathbb N} \mathbb Z$, which is a cotorsion group  \cite{Eda2016}.  For subsets of $\mathbb R^3$ it is also possible for a Peano continuum to have first homology isomorphic to $P/S$, for example Griffith's double cone \cite{EdaFischer2016}.  More complicated homology groups are possible for higher dimensional Euclidean spaces. For example Herfort and Hojka show that the shrinking wedge of projective planes has first homology group isomorphic to  $\bigl(\bigoplus_\mathbb N \mathbb Z/2\mathbb Z\big) \oplus P/S$ \cite{HerfortHojka17(2)}.  

A \emph{polyhedral expansion} for a Peano continuum $X$ is an inverse sequence of finite polyhedra $\{P_i\}$ such that $X =\varprojlim P_i$.   Let $\check H_1(X) = \varprojlim H_1(P_i)$.   This group is commonly referred to as the first \emph{\v{C}ech homology group} of $X$ even though it is not part of a classical homology theory because it does not always satisfy the long exact sequence of a pair.  
A \emph{cotorsion group} is an abelian group for which every extension by a torsion-free group splits, see (\ref{defn: cotorsion}).

\begin{thmm}[\ref{thmm: splitting}]
  If $X$ is a Peano continuum  with $\check H_1(X)$ torsion-free, then the first singular homology of $X$ splits as $H_1(X) = \check H_1(X)\oplus K$, where $K$ is the first homology shape kernel of $X$, a cotorsion group.
\end{thmm}

Since the shape homology of subset of $\mathbb R^3$ is torsion free by Alexander duality, we  have the following corollary.

\begin{corr}[\ref{cor: homology splitting}]
  If $X\subset \mathbb R^3$ is a Peano continuum, then $H_1(X) = \mathbb Z^\lambda\oplus K$, where $K$ is the first homology shape kernel of $X$, a cotorsion group, and $\lambda$ is a countable cardinal.
\end{corr}

The known examples follow a particular pattern, which leads us to pose the following:

\begin{conjj}
  If $X$ is a Peano continuum in $\mathbb R^3$, then the first homology shape kernel is isomorphic to $\displaystyle{\mathbb Z^{\mathbb N}/\bigoplus\limits_{i\in\mathbb N} \mathbb Z}$ or is trivial.   Equivalently $H_1(X)$ is isomorphic to $\mathbb Z^\lambda \oplus \left(\mathbb Z^{\mathbb N}/\bigoplus\limits_{i\in\mathbb N} \mathbb Z\right)$, for some countable cardinal $\lambda$, or $H_1(X)= \mathbb Z^\lambda$ for some finite cardinal $\lambda$.
\end{conjj}

In order to prove Theorem 1 (\ref{thmm: splitting}), we construct an infinite product of loops that are based at distinct points, see Section \ref{Sec: Main Construction}.  The set of basepoints for the loops could potentially be a countable dense subset.  Since this product will depend on the paths between the basepoints, this approach naturally leads to understanding of properties of the first homology group or subgroups of the fundamental group containing the commutator subgroup, which we study using commutator calculus.  

Shelah used an infinite word structure on fundamental groups to show that the rationals are not the fundamental group of any path-connected, locally path-connected, compact metric space \cite{Shelah1988}. 
Morgan and Morrison used the infinite words to show that the fundamental group of the Hawaiian earring embedded into the inverse limit of free groups \cite{mm}.    Eda introduced free $\sigma$-products in \cite{Eda92} as one approach to formalizing the word structure on the fundamental group of the Hawaiian earring. Cannon and Conner further studied the group structure of the Hawaiian earring in \cite{cc1}.

Infinite word arguments were used to show that every homomorphism between fundamental groups of one-dimensional or planar Peano continua is induced by continuous maps up to a change of basepoint isomorphism \cite{ConnerKent2019, eda, eda7, Kent18}, as well as to show how to recover the topological space from the fundamental group for one-dimensional Peano continua that are not locally simply connected at any point \cite{ce}.  


We use this infinite word structure built in Section \ref{Sec: Main Construction} to prove the following results about building cotorsion subgroups, which is the central step in our proof of Theorem \ref{thmm: splitting}.

\begin{thmm}[\ref{thm: higman complete}]
    Let $X$ be a Peano continuum and $H$ a subgroup of $\pi_1(X,x_0)$ containing the commutator subgroup.  Then $\overline H/ H$ is cotorsion, where $\overline H$ is the closure of $H$ in the shape topology.
\end{thmm}

We will refer to the closure, in the shape topology, of the commutator subgroup of the fundamental group as the \emph{infinite commutator subgroup}, see Definition \ref{defn:infinite commutator subgroup}.  

\begin{corr}
    The quotient of the infinite commutator subgroup by the commutator subgroup of the fundamental group of a Peano continuum is cotorsion.
\end{corr}




  Another application of our construction in Section \ref{Sec: Main Construction} is to show that certain pro-coverings are path-connected. (An inverse limit of covering spaces is a \emph{pro-covering}, see Definition \ref{def: pro-cover}.) In general, pro-coverings need not be path-connected.  Even simple examples give rise to non-path-connected pro-coverings, for instance  solenoids are inverse limits of covering spaces of a circle, see \cite{ConnerHerfortKentPavesic2021, ConnerHerfortPavesic2018} for additional examples.  Here we show that when considering covering spaces corresponding to subgroups containing the commutator subgroup of the fundamental group we are able to show the following.

\begin{thmm}[\ref{thm:path connected-closed}]
    Let $X$ be a Peano continuum.  Let $H$ be a closed subgroup of $\pi_1(X,x_0)$ that contains the commutator subgroup.  Then there exists a unique path-connected pro-covering $\rho:(E, e_0) \to (X, x_0)$ such that $\rho_*\bigl(\pi_1(E,e_0)\bigr)=H$.
\end{thmm}

\noindent This allows us to show that the path-connected pro-covering corresponding to the infinite commutator subgroup of the fundamental group is the minimal pro-covering that induces the trivial homomorphism on \v{C}ech cohomology, see Theorem \ref{thm:minimal Cech}.  We are also able to give a characterization of path-connectivity of pro-coverings in terms of maps to simplicial complexes, see Theorem \ref{thm: characterization in terms of maps}

Combining these results and the infinite word structure from Section \ref{Sec: Main Construction} allows us to provide a new proof of Eda and Kawamura's result that $H_1(X)$ surjects onto  $\check H_1(X)$ when $X$ is a Peano continuum \cite{EdaKawamura2000.3}.  Note, Eda and Kawamura actually show that the map is surjective whenever $X$ is a compact locally connected metric space.  The following result follows the definition of the shape topology and Eda and Kawamura's result.

\begin{thmm}[\ref{thm:short exact fundamental group}]
  If $X$ is a Peano continuum and $C$ is the commutator subgroup of $\pi_1(X,x_0)$, then we have the following short exact sequence of groups,
  \[0\to \overline C \to \pi_1(X,x_0) \to \check H_1(X) \to 0,\]  where $\overline C$ is the closure of $C$ in the shape topology and $\check H_1(X)= \varprojlim H_1(P_i)$, for a polyhedral expansion $\{P_i, \sigma_{i,j}\}$ of $X$.
\end{thmm}

\noindent This descends to the following short exact sequence for homology: \[0\to \overline C/C \to H_1(X,\Z) \to \check H_1(X) \to 0.\]


\section{Definitions and notations.}

\subsection*{Notational conventions.}  Let $\alpha, \beta, \gamma\colon [0,1]\to X$ be paths in $X$. We will write $\alpha=_p\beta$ if $\alpha$ is a reparametrization of $\beta$. We will write $\alpha =_C \beta$, if there exists a path $\gamma$ such that $\alpha$ is homotopic rel endpoints to $\beta*\gamma$ and $[\gamma]$ is in the commutator subgroup of the fundamental group. We will use the standard convention that  $ \overline \alpha$ is the path $\overline\alpha(t) = \alpha(1-t)$.  We will use $*$ to denote the standard concatenation of paths, i.e., $\alpha*\beta(t) = \alpha(2t)$ for $t\in[0,1/2]$ and $\alpha*\beta(t) = \beta(2t-1)$ for $t\in[1/2, 1]$.  For a loop $\beta$ and integer $n$, we will let $\beta^n $ be a parametrization of the $n$-fold concatenation of $\beta$ for $n>0$, the $|n|$-fold concatenation of $\overline \beta$ for $n<0$, and the constant path at $\beta(0)$ for $n=0$.

\begin{defn}\label{def: pro-cover}

    Let $X$ be a connected and locally path-connected space.  A \emph{covering sequence} is a sequence of covering maps $\bigl\{\rho_i:(E_i, e_i) \to (X, x_0)\bigr\}_{i\in\N}$ such that $\rho_{i*}\bigl(\pi_1(E_i,e_i)\bigr) \leq \rho_{j*}\bigl(\pi_1(E_j,e_j)\bigr)$ whenever $j \leq i$.  By the lifting theorem for covering spaces, for each pair $i,j$ with $j\leq i$, there exists a unique covering map $\sigma_{i,j}: (E_i,e_i) \to (E_j,e_j)$ such that $\rho_i = \rho_j\circ \sigma_{i,j}$. It follows that $\sigma_{j,k}\circ \sigma_{i,j} = \sigma_{i,k}$, whenever $k\leq j\leq i$.

    Thus a covering sequence uniquely produces an inverse sequence, $ \bigl\{E_i, \sigma_{i,j}\bigr\}$.  Let $E$ be the inverse limit of the inverse sequence $ \bigl\{E_i, \sigma_{i,j}\bigr\}$, which we will call a \emph{pro-covering space}.  The inverse limit of the maps $\rho_i:(E_i, e_i) \to (X, x_0)$ is a map  $\rho \colon E\to X$, which we will call a \emph{pro-covering map}.  Let $\sigma_i: (E, e_0) \to (E_i, e_i)$ be the projection map from $E$ to $E_i$. Notice that $\rho$ satisfies $\rho= \rho_i\circ \sigma_i$ for all $i$. Thus pro-covering maps satisfy the following commutative diagram.

\begin{center}
\begin{tikzcd}

     E_1\arrow{dd}{\rho_1}&E_2\arrow{l}{\sigma_{2,1}}\arrow{ldd}{\rho_2} & \cdots\arrow{l}{} &  E_i\arrow{l}{}\arrow{llldd}{\rho_i} & \cdots\arrow{l}{} & E\arrow{llllldd}{\rho}\arrow{l}\arrow[bend right]{ll}{\sigma_i}\arrow[bend right]{llll}{\sigma_2}\arrow[bend right]{lllll}{\sigma_1}\\
      &  & & &\\X &  & & &
\end{tikzcd}
\end{center}

\end{defn}

    Solenoids are well known examples of pro-coverings.  Joel Cohen studied inverse limits of principal fibrations in \cite{Cohen73}. Inverse limits of covering spaces were also studied in \cite{ConnerHerfortKentPavesic2021},  \cite{ConnerHerfortPavesic2018}, \cite{ConnerHerfortPavesic2020}, \cite{Hills2019}. The authors introduced that terminology \emph{pro-coverings} in \cite{ConnerHerfortKentPavesic2025}.

\subsection*{Polyhedral expansion}(see \cite[p.48]{MardesicSegalShape})  A \emph{polyhedral expansion} $\{P_i, \sigma_{i,j}\}$ of $X$ is an inverse sequence of simplicial complexes $P_i$ together with maps $\sigma_{i,j} \colon P_i\to P_j$, for $i> j$, such that $X = \varprojlim \{P_i, \sigma_{i,j}\}$.  We will use $\sigma_i\colon X \to P_i$ to denote the projection from $X$ to $P_i$.  Every compact metric space has a polyhedral expansion, see \cite[Theorem 4.10.10]{Sakai2013}.  When considering pointed functors, we will use $x_0$ to denote the fixed base point of $X$ and $p_i$ a fixed base point of $P_i$ such that $\sigma_i(x_0) = p_i$.
The projection maps, $\sigma_i$, induce a homomorphism  $\partial: \pi_1(X,x_0) \to \varprojlim \bigl\{\pi_1(P_i, p_i), \sigma_{i,j*}\bigr\}=\check \pi_1(X, x_0)$ into the \emph{(first) shape group} of $X$.  We will use $u_i: \varprojlim \bigl\{\pi_1(P_i, p_i), \sigma_{i,j*}\bigr\} \to \pi_1(P_i, p_i)$ to denote the projection homomorphism from the shape group to the fundamental group of $P_i$.

\begin{defn}\label{defn:shape topology}
    The \emph{shape topology} on $\pi_1(X,x_0)$ is the coarsest topology such that for every polyhedron $Q$ and every
continuous map $f\colon X\to Q$ the induced function $f_*\colon\pi_1(X,x_0)\to\pi_1(Q,q_0)$ is continuous, when $\pi_1(Q,q_0)$ is endowed  with the discrete topology.

    Throughout the remainder of the paper $X$ will be a Peano continuum and $\pi_1(X,x_0)$ will be considered a topological space equipped with the shape topology.

\end{defn}

The following lemma is Lemma A.20 in \cite{ConnerHerfortKentPavesic2025}.

\begin{lem}\label{lem:shape is shape}
  Let $X$ be a compact metric space and $\{P_i, \sigma_{i,j}\}$ be any polyhedral expansion of $X$.  Then $$\mathcal S = \bigl\{\sigma_{i*}^{-1} (O) \mid i\in \mathbb N \text{ and } O\subset \pi_1(P_i, p_i) \}$$

  is a subbasis for the shape topology on $\pi_1(X,x_0)$.

\end{lem}

Given a collection $\U$ of open subsets of $X$, the \emph{$\U$-Spanier subgroup} of $\pi_1(X, x_0)$ is the subgroup
$\pi_1(X, x_0; \U)$ of the fundamental group which is generated by homotopy classes of paths having a representative of the form $(\alpha * \beta)*\overline\alpha$ where $\beta$ is a loop contained in some element of $\U$ and  $\alpha$ is a path from $x_0$ to $\beta(0)$.  

For $\U$ a collection of open sets of $X$, we will let $2\U = \{ U\cup V\mid U,V\in \U\}$.  Recall that a \emph{covering subgroup} of the fundamental group is the image of the fundamental group of a covering space.

\begin{lem}\label{lem: factor}
    Let $X$ be a Peano continuum and $\{P_i, \sigma_{i,j}\}$ a polyhedral expansion of $X$.  For every covering subgroup $K$ of $\pi_1(X,x_0)$, there exists an $i$ such that $\ker(\sigma_{i*})\leq K$,  where $\sigma_{i*}: \pi_1(X,x_0) \to \pi_1(P_i, p_i)$ is the homomorphism induced by the projection map $\sigma_i: X\to P_i$.
\end{lem}

\begin{proof}
Since $K$ is a covering subgroup, there exists a finite open cover $\U$  of $X$ by path-connected open sets  such that $\pi_1(X,x_0;2\U)\subset K$.  By the proof of \cite[Theorem 5 in \S5.2 of Chapter I]{MardesicSegalShape}, there exists an $i$ and an open cover $\V$ of $P_i$ such that $\sigma_i^{-1}(\V)$ refines $\U$.  By  repeatedly barycentrically subdividing $P_i$, we may assume that the open star of any vertex is contained in an element of $\V$,  see \cite[Appendix 1.1, Theorem 4]{MardesicSegalShape}.  Then Lemma 3.4 in \cite{ConnerHerfortKentPavesic2021} implies that $\ker(\sigma_{i*})\leq\pi_1(X,x_0;2\U)\subset K$.
\end{proof}

An element $g$ of a group $G$ is \emph{infinitely divisible}, if there are infinitely many integers $n$ and corresponding nonidentity elements $g_n$ of $G$ so that $(g_n)^n = g$.  The following is a consequence of Theorem 4.4 from \cite{cc3}.

\begin{lem}\label{lem: covering subgroup}

   The kernel of every homomorphism from the fundamental group of a first countable space $X$ to a countable abelian group with no infinitely divisible elements is a covering subgroup.
\end{lem}

\begin{proof}
  Let $G$ be a countable abelian group with no infinitely divisible elements and  $\phi: \pi_1(X,x_0) \to G$ a homomorphism.  Fix $x\in X$ and choose a path $\alpha:[0,1]\to X$ from $x$ to $x_0$ and let $\hat\alpha: \pi_1(X,x)\to \pi_1(X,x_0)$ be the canonical change of basepoint isomorphism.  By applying Theorem 4.4 of \cite{cc3} to the homomorphism $\phi\circ\hat\alpha: \pi_1(X,x) \to G$, there exists a neighborhood $U$ of $x$ such that $\phi\circ\hat \alpha \circ i_*\big(\pi_1(U,x)\big)$ is trivial, where $i_*$ is the homomorphism induced by the inclusion map $i: U \to X$.  Note that this is independent of the choice of $\alpha$ in the sense that $\phi\circ\hat \beta \circ i_*\big(\pi_1(U,x)\big)$ will also be trivial for any other path $\beta:[0,1]\to X$ from $x$ to $x_0$.  Since $x$ was arbitrary,  Theorem 7.8 of \cite{cc3} implies that the kernel of $\phi$ is a covering subgroup.
\end{proof}


\section{Main Construction}\label{Sec: Main Construction}
In Section \ref{binarysequences} and \ref{construction subsection}, we construct a path that runs through a path $\gamma$ and, for each dyadic rational $t$, the path will traverse a closed loop based at $\gamma(t)$ a given number of times.  Varying the parameters, a plethora of paths can be produced which will help in Section \ref{cotorsion} to prove the cotorsion statement of Theorem \ref{thm: higman complete}.

\subsection{Binary sequences}\label{binarysequences} Let $\calS$ be the set of finite binary sequences with last term $1$.  There is a natural bijection from $\calS$ to the dyadic rationals in $(0,1)$ given by $(c_1, c_2, \cdots, c_{k-1}, c_k) \mapsto   \sum\limits_{i=1}^k \frac{c_i}{2^{i}}$.  (Note $c_k$ is always 1.)  When convenient, we will identify $(c_1, c_2, \cdots, c_{k-1}, c_k)$ with the dyadic rational $\sum\limits_{i=1}^k \frac{c_i}{2^{i}}$.  If $\bfs=(c_1, \cdots, c_{k-1},1)$ and $i\in\{0,1\}$, we will let $(\bfs,i) = (c_1, \cdots, c_{k-1}, i, 1)\in \mathcal S$.  We will say the length of $\bfs\in\calS$ is $k$ if $\bfs$ is a $k$-tuple and denote the length of $\bfs$ by $|\bfs|$. Thus the distance between $(\bfs,i)$ and $\bfs$ (as dyadic rationals) is $1/2^{|\bfs|+1}$.  If $\bfs = (c_1, c_2, \cdots, c_{k-1}, 1)$ and $ \bfs'= (c_1, c_2, \cdots, c_{i-1}, 1)$, for some $i<k$, we will say that $\bfs'$ \emph{precedes}  $\bfs$. For the sake of convenience, $(1)$ precedes all other elements of $\calS$.  Note this is not the standard linear order but the partial order determined by the dyadic branching tree with root $(1)$, see Figure \ref{fig: precedes}.  

\begin{figure}[h]
    \centering
    {\tiny\def\svgwidth{\textwidth}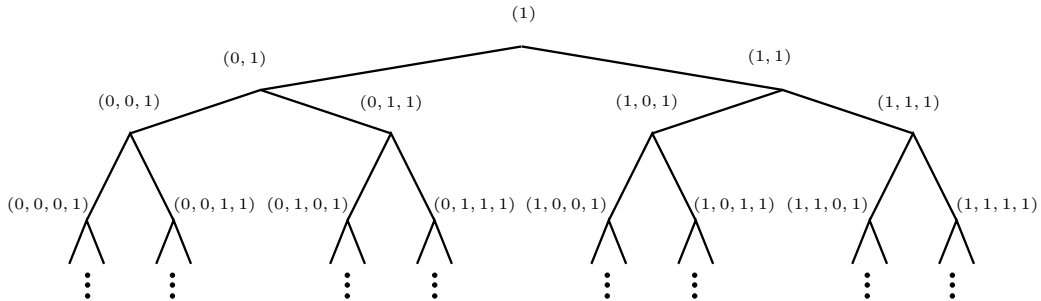}
    \caption{The partial order on $\calS$ is given by the partial ordering on this binary tree.}
    \label{fig: precedes}
\end{figure}

\subsection{Constructing nested open sets}\label{construction subsection}

Fix a sequence of natural numbers  $(n_i)$.

Let $B_{(1)}=(0,1)$ and $D_{(1)} = \left\{\frac{1}{n_1}, \ldots, \frac{n_1 -1}{n_1}\right\}$.  Let $\psi_{(1)}: B_{(1)}\backslash D_{(1)} \to (0,1)$ be the unique orientation preserving map which takes each component of $B_{(1)}\backslash D_{(1)}$ linearly onto (0,1).  Let $A_{(1)} = \psi_{(1)}^{-1}\bigl((\frac37, \frac47)\bigr)$ and $P^t_{(1)} = \psi_{(1)}^{-1}\bigl((\frac{2t}7, \frac{2t+1}7)\bigr)$ for $t\in\{0,1,2,3\}$. Each of the sets $A_{(1)}$, $P^t_{(1)}$ has exactly $n_1$ components of equal length.

For clarity, we will explicitly do the first inductive step before doing the generic situation.
For $i=0,1$, let $B_{(i,1)} = \psi_{(1)}^{-1}\bigl((\frac{4i+1}7, \frac{4i+2}7)\bigr)$.  Notice that each $B_{(i,1)}$ has exactly $n_1$ components of equal length.  Let $D_{(i,1)}$ consist of $(n_2 -1)$ distinct points from each component of $B_{(i,1)}$, which we will assume are evenly spaced in each component of $B_{(i,1)}$.  Let $\psi_{(i,1)}: B_{(i,1)} \backslash D_{(i,1)} \to (0,1)$ be the unique orientation preserving map which takes each component of $ B_{(i,1)} \backslash D_{(i,1)}$ linearly onto (0,1).  Let $A_{(i,1)} = \psi_{(i,1)}^{-1}\bigl((\frac37, \frac47)\bigr)$ and $P^t_{(i,1)} = \psi_{(i,1)}^{-1}\bigl((\frac{2t}7, \frac{2t+1}7)\bigr)$ for $t\in\{0,1,2,3\}$.   Each of these sets has exactly $n_1\cdot n_2$ components of equal length.
Notice that \[[0,1] =  \overline {B_{(1)}}=\overline{ P_{(1)}^0}\bigsqcup B_{(0,1)}\bigsqcup \overline{P_{(1)}^1}\bigsqcup A_{(1)}\bigsqcup \overline{P_{(1)}^2}\bigsqcup B_{(1,1)}\bigsqcup \overline{P_{(1)}^3}.\]

For any $E\in\bigr\{P_{(1)}^t,B_{(0,1)}, B_{(1,1)}, A_{(1)}\mid t\in\{0,1,2,3\}\bigl\}$, the map $\psi_{(1)}$ takes each component of $E$ linearly onto the same subinterval of $(0,1)$.  Thus, for any $E\in\bigr\{P_{(1)}^t,B_{(0,1)}, B_{(1,1)}, A_{(1)}\mid t\in\{0,1,2,3\}\bigl\}$, the map $\bigl(\psi_{(1)}|_{C'}\bigr)^{-1}\circ \psi_{(1)}|_{C} $ is a homeomorphism from $C$ to $C'$ for any two components $C, C'$ of $E$.

We will now do the generic inductive step of our construction. Suppose that by induction, we have defined $B_\bfs, D_\bfs, \psi_\bfs: B_\bfs\backslash D_\bfs \to (0,1)$, $P^t_{\bfs}$, and $A_\bfs$ for all $\bfs\in\calS$ with $|\bfs|\leq k$ and $t\in\{0,1,2, 3\}$ such that

\begin{enumerate}
  \item\label{nested} $B_{\bfs}\subset B_\bft$, whenever $\bft$ precedes $\bfs$,
  \item\label{componets} $B_\bfs$ has exactly $n_1\cdot n_2\cdot \ldots \cdot n_{|\bfs|-1}$ components (for $|\bfs|>1$),
  \item\label{D_s} $D_\bfs$ consists of $n_{|\bfs|}-1$ distinct points from each component of $B_\bfs$, 
  \item\label{linear} $\psi_{\bfs}: B_{\bfs} \backslash D_{\bfs} \to (0,1)$ is the unique orientation preserving map which takes each component of $ B_{\bfs} \backslash D_{\bfs}$ linearly onto (0,1), and
 
  \item\label{subintervals} for $|\bfs|<k$, $\overline {B_{\bfs}} = \overline{ P_{\bfs}^0}\bigsqcup B_{(\bfs,0)}\bigsqcup \overline{P_{\bfs}^1}\bigsqcup A_{\bfs}\bigsqcup \overline{P_{\bfs}^2}\bigsqcup B_{(\bfs,1)}\bigsqcup \overline{P_{\bfs}^3}$.

\end{enumerate}

We will now define the required sets and functions for binary sequences of length $k+1$.  Let $\bfs$ be a binary sequence of length $k$. For $i=0,1$, let $B_{(\bfs,i)} = \psi_{\bfs} ^{-1}\bigl((\frac{4i+1}7, \frac{4i+2}7)\bigr)$, which has exactly $n_1\cdot n_2\cdot \ldots\cdot n_{k}$ components.   Let $D_{(\bfs,i)}$ consist of $(n_{k+1} -1)$ distinct points from each component of $B_{(\bfs,i)}$, which we will assume are evenly spaced in each component of $B_{(\bfs,1)}$.  Then $B_{(\bfs,i)} \backslash D_{(\bfs,i)}$ has exactly $n_1\cdot n_2\cdot \ldots \cdot n_{k+1}$ components.  Let $\psi_{(\bfs,i)}: B_{(\bfs,i)} \backslash D_{(\bfs,i)} \to (0,1)$ be the unique orientation preserving map which takes each component of $ B_{(\bfs,i)} \backslash D_{(\bfs,i)}$ linearly onto (0,1).  Let $A_{(\bfs,i)} = \psi_{(\bfs,i)}^{-1}\bigl((\frac37, \frac47)\bigr)$ and $P^t_{(\bfs, i)} = \psi_{(\bfs,i)}^{-1}\bigl((\frac{2t}7, \frac{2t+1}7)\bigr)$ for $t\in\{0,1,2,3\}$. 
It is immediate that the new sets and functions satisfy conditions (\ref{nested})-(\ref{subintervals}).

\subsection{Constructing the path}\label{construction}

Let $X$ be a topological space with a fixed basepoint $x_0$.  Throughout Section \ref{construction}, we will fix 
\begin{enumerate}[(i)]
    \item\label{i} a sequence $(n_i)$  of nonzero integers,
    \item a nullhomotopic closed curve $\gamma:[0,1]\to X$ based at $x_0$, and
    \item\label{iii} loops $\alpha_\bfs:[0,1]\to X$ where $\alpha_\bfs$ is based at $\gamma ( \bfs )$ such that the diameters of $\im(\alpha_\bfs)$ converge to 0 as $|\bfs|$ diverges.
\end{enumerate}

With data (\ref{i}) - (\ref{iii}), we will construct a sequence of paths $\beta_i$ such that 
\begin{equation}\label{equation: beta_i}
    \beta_i =_C\left(\stackrel[{|\bfs|= j}]{}{\ast}(\eta_{\bfs}*\alpha_{\bfs}*\overline{\eta_{\bfs}}) \right)^{n_j}*(\beta_{i+1})^{n_j},
\end{equation} for some paths $\eta_{\bfs}$ from $x_0$ to $\gamma(\bfs)$. 

\begin{rmk}\label{rem: independent of base paths}
    Since we are only considering the relation $=_C$, this is independent of our choice of paths $\eta_\bfs$, in the sense, if $\beta_i$ satisfies Equation \ref{equation: beta_i} for some choice of paths $\eta_{\bfs}$, then  $\beta_i =_C\left(\stackrel[{|\bfs|= j}]{}{\ast}(m_{\bfs}*\alpha_{\bfs}*\overline{m_{\bfs}}) \right)^{n_j}*(\beta_{j+1})^{n_j}$ for any choice of paths $m_\bfs$ from $\gamma (0)$ to $\gamma(\bfs)$.
\end{rmk} 

To simplify our notation, we will assume that $(n_i)$ is a sequence of natural numbers.  To do the generic case, simply replace $\alpha_\bfs$ with $\overline\alpha_\bfs$ whenever $n_{|\bfs|}$ is negative.

Let $\beta_{1,1}(t) = \begin{cases}
   \alpha_{(1)}\bigl(7\psi_{(1)}(t) - 3\bigr) , & \mbox{if }  t\in A_{(1)} \\
   \gamma(\bfs), & \mbox{if } t\in B_\bfs \mbox{ and } |\bfs|=2\\
   \gamma \left(\frac{7\psi_{(1)}(t) - i}{4}\right), & \mbox{if } t\in \overline{P_{(1)}^i}.
 \end{cases}$

In other words, $\beta_{1,1}$ is a parametrization of $\gamma|_{\left[\frac{i}{4}, \frac{i+1}{4}\right]}$ on each component of $P_{(1)}^i $, a parametrization of $\alpha_{(1)}$ on each component of $A_{(1)}$, and the constant loop at $\gamma(\bfs)$ on each component of $B_{(1,0)}$ and $B_{(1,1)}$.  In particular, \[
  \beta_{1,1} =_p  \left(\gamma|_{[0,\frac14]}*\gamma(0,1)*\gamma|_{[\frac14,\frac12]} *\alpha_{(1)}* \gamma|_{[\frac12,\frac34]}*\gamma(1,1)*\gamma|_{[\frac34,1]}\right)^{n_1}. \]
(Recall we are identifying $(c_1, c_{2}, \cdots, c_k)$ with dyadic rational $\sum\limits_{i=1}^k \frac{c_i}{2^{i}}$.  So $\gamma(0,1) = \gamma(\frac14)$ and $\gamma (1,1) = \gamma (\frac34)$.  In this setting, we are using $\gamma (\bfs)$ to denote the constant path at $\gamma(\bfs)$.) Since $\gamma$ is nullhomotopic, this means that
\[
  \beta_{1,1} =_C  (\eta_{(1)}*\alpha_{(1)}*\overline{\eta_{(1)}})^{n_1}\]
where $\eta_\bfs$ is any path from $\gamma(0)$ to $\gamma(\bfs)$.

Let $\delta_{(\bfs,0)}:[0,1]\to X$ be a reparametrization of $\gamma|_{[(\bfs,0), \bfs]}$ and  $\delta_{(\bfs,1)}:[0,1]\to X$ be a reparametrization of $\gamma|_{[\bfs, (\bfs,1)]}$.

Then we will inductively define $\beta_{1,k}:[0,1]\to X$ by  \[\beta_{1,k}(t) = \begin{cases}
   \beta_{1,k-1}(t) , & \mbox{if }  t\not\in B_\bfs \mbox{ and } |\bfs|=k \\
   \alpha_{\bfs}\bigl(7\psi_{\bfs}(t) - 3\bigr) , & \mbox{if }  t\in A_{\bfs},  \mbox{ and } |\bfs|=k\\
   \gamma(\bfs), & \mbox{if } t\in B_\bfs \mbox{ and } |\bfs|=k+1\\
   \delta_{(\bfs,i)} \bigl(4i -7\psi_\bfs(t)\bigr), & \mbox{if } t\in \overline{P_{\bfs}^{2i}}  \mbox{ and } |\bfs|=k\\
   \delta_{(\bfs,i)} \bigl(4i+3 -7\psi_\bfs(t)\bigr), & \mbox{if } t\in \overline{P_{\bfs}^{2i+1}}  \mbox{ and } |\bfs|=k.
 \end{cases}\]

In other words, $\beta_{1,k}$ agrees with $\beta_{1, k-1}$ on $[0,1]\backslash \bigcup\limits_{|\bfs| = k} B_\bfs$ and is  a parametrization of $\alpha_{\bfs}$ on each component of $A_{\bfs}$ for $|\bfs|=k$, the constant loop at $\gamma(\bfs)$ on each component of $B_\bfs$ for $|\bfs|=k+1$, a parametrization of $\delta_{(\bfs,i)}$ on each component of $P_{\bfs}^{2i} $, and a parametrization of $\overline{\delta_{(\bfs,i)}}$ on each component of $P_{\bfs}^{2i+1} $.

For example, \[
  \beta_{1,2} =_p  \left(\gamma|_{[0,\frac14]}*\vartheta_1 *\gamma|_{[\frac14,\frac12]} *\alpha_{(1)}* \gamma|_{[\frac12,\frac34]}*\vartheta_2*\gamma|_{[\frac34,1]}\right)^{n_1}. \]
where
  \begin{align*}
    \vartheta_1= & \left(\delta_{(0,0,1)}*\gamma(0,0,1)*\overline\delta_{(0,0,1)} *\alpha_{(0,1)}* \delta_{(0,1,1)}*\gamma(0,1,1)*\overline\delta_{(0,1,1)}\right)^{n_2} \\
    \vartheta_2= & \left(\delta_{(1,0,1)}*\gamma(1,0,1)*\overline\delta_{(1,0,1)} *\alpha_{(1,1)}* \delta_{(1,1,1)}*\gamma(1,1,1)*\overline\delta_{(1,1,1)}\right)^{n_2}.
  \end{align*}

Again, this means that
\[
  \beta_{1,2} =_C  (\eta_{(1)}*\alpha_{(1)}*\overline{\eta_{(1)}})^{n_1}*(\eta_{(0,1)}*\alpha_{(0,1)}*\overline{\eta_{(0,1)}})^{n_1\cdot n_2}*(\eta_{(1,1)}*\alpha_{(1,1)}*\overline{\eta_{(1,1)}})^{n_1\cdot n_2} \]
where $\eta_\bfs$ is any path from $\gamma(0)$ to $\gamma(\bfs)$.

The map $\beta_{1,k}$ traverses the same path on each component of $B_\bfs$, for $\bfs$ with $|\bfs|= k$ (since on each part of each component of $B_\bfs$ the map $\beta_{1,k}$ factors through the maps $\psi_\bfs$ or is constant).  Thus, by induction, we can see that $\beta_{1,k}$ traverses the same path on each component of $B_\bfs$, whenever $|\bfs|\leq k$.  Thus
\[
  \beta_{1,k} =_C  \left(\stackrel[{|\bfs|\leq k}]{}{\ast}(\eta_{\bfs}*\alpha_{\bfs}*\overline{\eta_{\bfs}})^{n_1\cdot \ldots \cdot n_{|\bfs|}} \right).\]

Since diameters of $\alpha_\bfs$ and $\delta_{(\bfs, i)}$ both converge to $0$ as $|\bfs|$ diverges, $\beta_{1,k}$ converges uniformly to a path $\beta_1$. Since $\beta_{1,k}$ traverses the same path on each component of $B_\bfs$ whenever $|\bfs|\leq k$, the limit $\beta_1$ traverses the same path on each component of $B_\bfs$ for all $\bfs$.

For each $\bfs$, fix a component $C_\bfs$ of $B_\bfs$.  For any $k$, we have that
\[
  \beta_1 =_C  \left(\stackrel[{|\bfs|\leq k}]{}{\ast}(\eta_{\bfs}*\alpha_{\bfs}*\overline{\eta_{\bfs}})^{n_1\cdot \ldots \cdot n_{|\bfs|}} \right)*\left(\stackrel[{|\bfs|= k+1}]{}{\ast} (\eta_{\bfs}*\beta_1|_{C_\bfs}*\overline{\eta_{\bfs}})^{n_1\cdot \ldots \cdot n_{k}} \right). \]

\begin{equation}\label{equation: >= 2}
    \text{For } j\geq 2, \text{ let }\beta_j(t) = \begin{cases}
    \gamma(\bfs), & \mbox{if } t\in B_\bfs \backslash C_\bfs \mbox{ and } |\bfs|=j\\
    \gamma(\bfs), & \mbox{if } t\in A_\bfs  \mbox{ and } |\bfs|<j\\
    \beta_1(t) , & \mbox{otherwise}.
  \end{cases}
\end{equation}

Thus, we have that
\begin{align*}
  \beta_j =_C & \stackrel[{|\bfs|= j}]{}{\ast} (\eta_{\bfs}*\beta_1|_{C_\bfs}*\overline{\eta_{\bfs}}) \\
  =_C &  \left(\stackrel[{|\bfs|= j}]{}{\ast}(\eta_{\bfs}*\alpha_{\bfs}*\overline{\eta_{\bfs}})^{n_j} \right)*\left(\stackrel[{|\bfs|= j+1}]{}{\ast} (\eta_{\bfs}*\beta_1|_{C_\bfs}*\overline{\eta_{\bfs}})^{n_j} \right)\\
  =_C &  \left(\stackrel[{|\bfs|= j}]{}{\ast}(\eta_{\bfs}*\alpha_{\bfs}*\overline{\eta_{\bfs}}) \right)^{n_j}*(\beta_{j+1})^{n_j}
\end{align*}

 Notice that by construction, $[\beta_j]$ converges to the identity element of $\pi_1\bigl(X,\gamma(0)\bigr)$ under the shape topology.

\subsection{Some properties of the paths $\beta_j$}\label{properties of construction}

\begin{lem}\label{simple equations}   
    Let $(n_i)$ be a sequence of nonzero integers.  Suppose that $\alpha_i:[0,1]\to X$ is a null sequence of loops in $X$ based at $x_1$. Then there exists a null sequence of loops $\beta_i$ at $x_1$ such that
\[\beta_i =_C(\alpha_i)^{n_i}*(\beta_{i+1})^{n_i}=_C(\alpha_i*\beta_{i+1})^{n_i}.\]
\end{lem}

\begin{proof}
Let $(n_i)$ and $(\alpha_i)$ be as in the statement of the lemma.  Let $\gamma$ be the constant path at $x_1$.   If $\bfs =(1, \cdots, 1)$ is the constant sequence of length $i$, then let $\alpha_\bfs = \alpha_i$.  If any coordinate of $\bfs\in\calS$ is $0$, then let $\alpha_\bfs$ be the constant path at $x_1$.  We can then use the construction in Section \ref{construction} to build a sequence of paths.

 Then $\beta_i =_C\Bigl(\stackrel[{|\bfs|= i}]{}{\ast}(\eta_{\bfs}*\alpha_{\bfs}*\overline{\eta_{\bfs}}) \Bigr)^{n_i}*(\beta_{i+1})^{n_i}$, for any choice of paths $\eta_{\bfs}$ for $x_1$ to $\gamma(\bfs)$.   

 As noted in Remark \ref{rem: independent of base paths}, 
 Since we are only considering this  up to the equivalence relation $=_C$, we may choose $\eta_\bfs$ to be the constant path at $x_1$.  Then \[\beta_i =_C   \Bigl(\stackrel[{|\bfs|= i}]{}{\ast}(\alpha_{\bfs}) \Bigr)^{n_i}*(\beta_{i+1})^{n_i}  =_C  (\alpha_{i})^{n_i}*(\beta_{i+1})^{n_i}.\]

Notice that by the construction,  $\beta_1$ traverses $\alpha_\bfs= \alpha_{|\bfs|}$ on each component of $A_\bfs$ and traverses some part of $\gamma$ on all other intervals.  Since $\gamma$ is the constant path at $x_1$, $\diam \big(\im(\beta_1)\big) \leq 2\max\limits_{i\in\mathbb N} \big\{ \diam \big(\im(\alpha_i)\big)\big\}$.  

By Equation \ref{equation: >= 2}, $\diam \big(\im(\beta_j)\big) \leq 2\max\limits_{j\in\mathbb N_{\geq j}} \big\{ \diam \big(\im(\alpha_i)\big)\big\}$, where $\mathbb N_{\geq j} =  \{j, j+1, j+2, \cdots\}$.  Since $(\alpha_i)$ is a null sequence of loops, so is $(\beta_i)$.

\end{proof}

\begin{lem}\label{higman solutions}
Let $\gamma$ be a nullhomotopic curve at $x_0$ and $(n_i)$ be a sequence of nonzero integers.  Suppose that for every $i\in\mathbb N$ and every $\bfs\in\calS$ with $|\bfs|\leq i$, the path $\bfsi:[0,1]\to X$ is a closed curve at $\gamma(\bfs)$ with diameter at most $4l_i$.  If $l_i$ converges to 0, then there exists a sequence of closed loops $\beta_i$ at $x_0$ such that
\[\beta_i =_C\left(\stackrel[{|\bfs|\leq i}]{}{\ast}(\eta_{\bfs}*\bfsi*\overline{\eta_{\bfs}}) \right)^{n_i}*(\beta_{i+1})^{n_i}.\]
where $\eta_\bfs$ is any path from $\gamma(0)$ to $\gamma(\bfs)$.

In addition, $[\beta_i]$ converges to the $[x_0]$ in $\pi_1(X,x_0)$ with the shape topology.
\end{lem}

\begin{proof}

  Fix $\bfs\in\calS$ and let $k= |\bfs|$.  Consider the sequence of integers $(1, n_{k+1}, n_{k+2}, \cdots )$ and the sequence of loops ${}_{k} b_\bfs, {}_{k+1} b_\bfs, {}_{k+2} b_\bfs, \cdots$ all based at $\gamma(\bfs)$.  Then, applying Lemma \ref{simple equations} to this sequence of loops and sequence of integers, we obtain a null sequence of paths $\{{}_i\alpha_{\bfs}\}$, all based at $\gamma(\bfs)$, such that ${}_1\alpha_\bfs = _C  ({}_{k}b_\bfs) *({}_{2}\alpha_\bfs) $ and, for $i>1$, ${}_i\alpha_\bfs = _C  ({}_{k+i-1}b_\bfs)^{n_{k+i-1}} *({}_{i+1}\alpha_\bfs)^{n_{k+i-1} }$.  For notational simplicity, we will shift our index on the sequence of $({}_i\alpha_\bfs)$ to start at $k$, i.e.,  we have a sequence of paths $({}_i\alpha_\bfs)_{i\geq k}$ with ${}_k\alpha_\bfs = _C  ({}_{k}b_\bfs) *({}_{k+1}\alpha_\bfs) $ and, for $i>k$, ${}_i\alpha_\bfs = _C  ({}_ib_\bfs)^{n_{i}} *({}_{i+1}\alpha_\bfs)^{n_i}$ for all $i\geq k$.  Let $\alpha_\bfs ={}_{|\bfs|}\alpha_\bfs$.  Notice that ${}_i\alpha_\bfs$ has diameter at most $8l_{|\bfs|}$.

  We will use the construction in Section \ref{construction}, this time using the nullhomotopic curve $\gamma$, the loops $\alpha_\bfs$, and the sequence $(n_i)$.  This gives us a sequence of curves $\beta_i'$ such that
  $\beta_i' =_C \bigl(\stackrel[{|\bfs|= i}]{}{\ast}(\eta_{\bfs}*\alpha_{\bfs}*\overline{\eta_{\bfs}})^{n_{i}} \bigr)*(\beta_{i+1}')^{n_i}$.   Let $\beta_i =  \left(\stackrel[{|\bfs|< i}]{}{\ast}(\eta_{\bfs}*{}_{i}\alpha_\bfs*\overline{\eta_{\bfs}}) \right)*\beta_i' $.

  Then, for fixed $i$,
  \begin{align*}
    \beta_i =_C & \bigl(\stackrel[{|\bfs|< i}]{}{\ast}(\eta_{\bfs}*{}_{i}\alpha_\bfs*\overline{\eta_{\bfs}}) \bigr)*\beta_i' \\
    =_C &  \bigl(\stackrel[{|\bfs|< i}]{}{\ast}(\eta_{\bfs}*{}_{i}\alpha_\bfs*\overline{\eta_{\bfs}}) \bigr)* \bigl(\stackrel[{|\bfs|= i}]{}{\ast}(\eta_{\bfs}*\alpha_{\bfs}*\overline{\eta_{\bfs}}) \bigr)^{n_{i}}*(\beta_{i+1}')^{n_i} \\
    =_C & \bigl(\stackrel[{|\bfs|< i}]{}{\ast}(\eta_{\bfs}*{}_{i}\alpha_\bfs*\overline{\eta_{\bfs}}) \bigr)*\bigl(\stackrel[{|\bfs|= i}]{}{\ast}(\eta_{\bfs}*\bigl( ({}_ib_\bfs) *({}_{i+1}\alpha_\bfs)\bigr)*\overline{\eta_{\bfs}}) \bigr)^{n_{i}} * (\beta_{i+1}')^{n_i}\\
    =_C & \bigl(\stackrel[{|\bfs|< i}]{}{\ast}(\eta_{\bfs}*{}_{i}\alpha_\bfs*\overline{\eta_{\bfs}}) \bigr)* \bigl(\stackrel[{|\bfs|= i}]{}{\ast}(\eta_{\bfs}*({}_ib_\bfs)*\overline{\eta_{\bfs}}) \bigr)^{n_{i}}* \bigl(\stackrel[{|\bfs|= i}]{}{\ast}(\eta_{\bfs}*{}_{i+1}\alpha_\bfs*\overline{\eta_{\bfs}}) \bigr)^{n_i} * (\beta_{i+1}')^{n_i}\\
   =_C & \bigl(\stackrel[{|\bfs|< i}]{}{\ast}(\eta_{\bfs}*({}_ib_\bfs)^{n_{i}} *({}_{i+1}\alpha_\bfs)^{n_i}*\overline{\eta_{\bfs}}) \bigr)*\bigl(\stackrel[{|\bfs|= i}]{}{\ast}(\eta_{\bfs}*({}_ib_\bfs)*\overline{\eta_{\bfs}}) \bigr)^{n_{i}}\\
   &\hspace{2in}* \bigl(\stackrel[{|\bfs|= i}]{}{\ast}(\eta_{\bfs}*{}_{i+1}\alpha_\bfs*\overline{\eta_{\bfs}}) \bigr)^{n_i}* (\beta_{i+1}')^{n_i}\\
   =_C & \bigl(\stackrel[{|\bfs|\leq i}]{}{\ast}(\eta_{\bfs}*({}_ib_\bfs)*\overline{\eta_{\bfs}}) \bigr)^{n_{i}}* \bigl(\stackrel[{|\bfs|< i+1}]{}{\ast}(\eta_{\bfs}*{}_{i+1}\alpha_\bfs*\overline{\eta_{\bfs}}) \bigr)^{n_i}*(\beta_{i+1}')^{n_i}\\
   =_C &\bigl(\stackrel[{|\bfs|\leq i}]{}{\ast}(\eta_{\bfs}*({}_ib_\bfs)*\overline{\eta_{\bfs}}) \bigr)^{n_{i}}*(\beta_{i+1})^{n_i}.
  \end{align*}

  This shows that $\beta_i$ satisfies the desired equation.  Recall that each loop  ${}_i\alpha_\bfs$ has diameter at most $8l_{|\bfs|}$.  Thus $[\beta_i]$ converges to the identity element in $\pi_1(X,x_0)$.

  \end{proof}


\section{Path-connected pro-coverings }

\begin{defn}
  Let $G$ be a group and $(H_n)$ a descending sequence of subgroups of $G$, i.e., $H_n\supset H_{n+1}$ for every $n$.  We will say that $G$ is \emph{complete relative to $(H_n)$} if every descending sequence of cosets $(H_ng_n)$ has nonempty intersection.  We invite the interested reader to read Section 3.1 of \cite{ConnerHerfortKentPavesic2025} for justification of this terminology.
\end{defn}

\begin{thm}[Theorem 3.22 in \cite{ConnerHerfortKentPavesic2025}]\label{complete imlplies path-connected}
  Suppose that $\rho\colon (E, e_0)\to (X, x_0)$ is a pro-covering.  Then $E$ is path-connected if and only if $\pi_1(X,x_0)$ is complete relative to $(H_n)$ where $H_n = \rho_{n*}\bigl(\pi_1(E_n,e_n)\bigr)$.
\end{thm}

We will require the following lemma.

\begin{lem}\label{lem: right basepoints}
    Let $\gamma: [0,1]\to X$ be a nullhomotopic space filling curve and $\U$ an open cover of $X$.  Fix $i\in \mathbb N$ and let $\epsilon =  \mesh\bigl(\mathcal U\bigr) + \max\limits_{1\leq j\leq 2^i} \diam\bigl(\im \bigl(\gamma|_{[\frac{j-1}{2^i}, \frac{j}{2^i}]}\bigr) \bigr)$.  For every loop $\alpha: [0,1]\to X$ with  $[\alpha]\in \pi_1(X, x_0; \U)$,  there exists, for each $j\in\{0, 1, \cdots, 2^i-1\}$, a loop $b_{j}: [0,1]\to X$ based at $\gamma(j/2^i)$ with diameter at most $2\epsilon$, such that \[\alpha =_C b_{0}*\gamma|_{[0, 1/2^i]} *b_{1} * \gamma|_{[1/2^i, 2/2^i]} * b_{2} * \cdots * b_{2^i-1} *\gamma|_{[(2^{i} -1)/2^i, 1]}.\]
\end{lem}

\begin{proof}
    Since $[\alpha]\in \pi_1(X, x_0; \U)$, there exists loops $c_l: [0,1]\to X$ and paths $a_l:[0,1]\to X$ such that $\alpha$ is  homotopic to $ a_1 * c_1 * \overline a_1* \cdots * a_k * c_k * \overline a_k$ and  $c_l$ has diameter at most the mesh of $\mathcal U$.  The loop $c_l$ is based at some point $x_l$ of $X$ and there exists $t_l\in [0,1)$ such that $\gamma (t_l) = x_l$.  Then $ c_l' = \gamma|_{\bigl[\lfloor 2^i \cdot t_l\rfloor/2^i, t_l\bigr]} *c_l*\overline{\gamma|_{\bigl[\lfloor 2^i \cdot t_l\rfloor/2^i, t_l\bigr]}}$  has diameter at most $\epsilon$.  (Recall $\lfloor x\rfloor$ is greatest integer less than or equal to $x$.)  Let $b_j$ be the concatenation of all $c_l'$ such that $\lfloor 2^i \cdot t_l\rfloor = j$.  Then $b_j$ has diameter less than $2\epsilon$ and $\alpha =_C b_{0}*\gamma|_{[0, 1/2^i]} *b_{1} * \gamma|_{[1/2^i, 2/2^i]} * b_{2} * \cdots * b_{2^i-1} *\gamma|_{[(2^{i} -1)/2^i, 1]}$.
\end{proof}

\begin{lem}\label{complete rel H_n}
    Let $\{P_i, \sigma_{i,j}\}$ be a polyhedral expansion of $X$ with projection maps $\sigma_i : X \to P_i$.  If $H\leq \pi_1(X,x_0)$ contains the commutator subgroup, then $\pi_1(X,x_0)$ is complete relative to $(K_n)$, where $K_n = H\cdot \ker(\sigma_{n*})$.
\end{lem}

\begin{proof}
  Let $(K_ng_n)$ be a descending sequence of cosets.  We need only show that $\bigcap\limits_{n\in\mathbb N} K_ng_n$ is non-empty. Notice have $K_1g_1\supset K_2g_2$ which implies $g_2g_1^{-1} = a_1k_1$ for some $a_1\in H$ and some $k_1\in  \ker(\sigma_{1*})$.  Thus $K_2g_2 = K_2(k_1g_1)$.  Inductively, we will assume that $K_ig_i = K_i(k_{i-1}\cdots k_1g_1)$, where $k_j\in \ker(\sigma_{j*})$.  Then  $K_{i+1}g_{i+1} \supset K_ig_i = K_i(k_{i-1}\cdots k_1g_1)$, where $k_j\in \ker(\sigma_{j*})$, which implies that $g_{i+1}(k_{i-1}\cdots k_1g_1)^{-1} = a_{i}k_i$ for some $a_i\in H$ and some $k_i\in  \ker(\sigma_{i*})$.  Thus $K_{i+1}g_{i+1} = K_{i+1}(k_{i}\cdots k_1g_1)$.

  By Lemma \ref{lem: factor} each element $k_i$ can be represented by an element of a Spanier subgroup of a cover with mesh at most $l_i$ where $l_i$ converges to 0.   Fix $\gamma:[0,1]\to X$ a nullhomotopic space filling curve.  Lemma \ref{lem: right basepoints} allows us to find loops ${}_ib_\bfs$ based at $\gamma(\bfs)$, for $|\bfs|\leq i$, with diameter at most $2l_i$ such that $k_i = [h_i] $ where $h_i =_C \stackrel[{|\bfs|\leq i }]{}{\ast} (\eta_\bfs*{}_ib_\bfs*\overline\eta_\bfs)$ for any choice of paths $\eta_\bfs$ from $x_0$ to $\gamma(\bfs)$.

  By Lemma \ref{higman solutions} (letting $n_i=1$ for all $i$), there exists loops $\beta_i$ such that
  \[\beta_i =_C\left(\stackrel[{|\bfs|\leq i}]{}{\ast}(\eta_{\bfs}*\bfsi*\overline{\eta_{\bfs}}) \right)*(\beta_{i+1})=_C h_i* \beta_{i+1},\]

  which implies that $\beta_1 =_C h_1*h_2*\cdots *h_i*\beta_{i+1} =_C \beta_{i+1}* h_i*h_{i-1}*\cdots *h_1$ for all $i$.

  Let $\alpha$ be a representative of $g_1$.  Fix an $i\in\mathbb N$.  Then there exists a $j\geq i$ such that $[\beta_j]\in \ker(\sigma_{i*})\subset K_i$.  Since the commutator subgroup is contained in $K_i$, we have that $ K_i[\beta_1*\alpha] = K_i [\beta_j*h_{j-1}*\cdots * h_1*\alpha] = K_i [h_{i-1}*\cdots *h_1*\alpha]= K_i(k_{i-1}\cdots k_1g_1)=K_{i}g_{i}$.  Thus $[\beta_1*\alpha]\in K_{i}g_{i}$ for all $i$, as desired.

\end{proof}

\begin{thm}\label{thm:path connected-closed}
    Let $X$ be a Peano continuum.  Let $H$ be a closed subgroup of $\pi_1(X,x_0)$ which contains the commutator subgroup.  Then there exists a unique path-connected pro-covering $\rho:(E, e_0) \to (X, x_0)$ such that $\rho_*\bigl(\pi_1(E,e_0)\bigr)=H$.  
\end{thm}

\begin{proof}
    The uniqueness conclusion will follow from Theorem 2 of \cite{ConnerHerfortKentPavesic2025}.

    Suppose that $\{P_i, \sigma_{i,j}\}$ is a polyhedral expansion of the Peano continuum $X$.  Let $\rho_i:(E_i, e_i) \to (X, x_0)$ be the covering space of $X$ such that $\rho_{i*} \bigl(\pi_1(E_i, e_i)\bigr) = H\cdot \ker(\sigma_{i*})$.  Let $E =  \varprojlim\bigl\{ E_i, \nu_{i,j}\bigr\}$ where $\nu_{i,j}: E_i \to E_j$ is the induced covering map.  Thus $E$ is a pro-covering of $X$.  By Lemma \ref{complete rel H_n} and Theorem \ref{complete imlplies path-connected}, $E$ is path-connected.

    A loop in $X$ at $x_0$ will lift to a loop in $E$ if and only if it lifts to a loop in each $E_j$.  Since  $H$ is closed, $\rho_*\bigl(\pi_1(E,e_0)\bigr) = \bigcap\limits_n \rho_{n*} \bigl(\pi_1(E_n, e_n)\bigr) = \bigcap\limits_n H\cdot \ker(\sigma_{n*})= H$.

\end{proof}

\begin{defn}\label{defn:infinite commutator subgroup}
  The $\Z$-kernel of $\pi_1(X,x_0)$ is the intersection of all kernels of homomorphisms from $\pi_1(X,x_0)$ to $\mathbb Z$, see \cite{ConnerHerfortKentPavesic18}.  The \emph{infinite commutator subgroup} is the closure (in the shape topology) of the commutator subgroup of the fundamental group.
\end{defn}


\begin{cor}\label{cor:infinite commutator subgrup}
    Let $X$ be a Peano continuum.  Then there exists a unique path-connected pro-covering over $X$ whose fundamental group is the $\Z$-kernel or the infinite commutator subgroup of  $\pi_1(X,x_0)$.

\end{cor}

\begin{proof}
    Since both the $\Z$-kernel and the infinite commutator subgroup contain the commutator subgroup and the infinite commutator subgroup is closed by definition, we need only show that the $\Z$-kernel is closed.  Every homomorphism from $\pi_1(X,x_0)$ to $\Z$  has an open kernel by Lemma \ref{lem: covering subgroup}.  Since open subgroups are also closed, every homomorphism from $\pi_1(X,x_0)$ to $\Z$ has a closed kernel.  The $\Z$-kernel is the intersection of kernels of all homomorphisms to $\Z$, each of which is closed.   So the $\Z$-kernel is the intersection of closed subgroups and hence is closed.
\end{proof}

\begin{lem}\label{lem: continuous lift}
    Suppose that $\rho:(E, e_0)\to (X,x_0)$ and  $\tilde \rho:(\tilde E, \tilde e_0)\to (X,x_0)$ are two path-connected pro-coverings over a Peano continuum $X$.  Then there exists a map $\theta: E\to \tilde E$ such that $\rho= \tilde \rho\circ\theta$ if and only if $\rho_*\big(\pi_1(E,e_0)\bigr)\subset \tilde \rho_*\big(\pi_1(\tilde E,\tilde e_0)\bigr).$
\end{lem}

    \begin{proof}
        The forward direction is immediate and we need only prove that if $\rho_*\big(\pi_1(E,e_0)\bigr)\subset \tilde \rho_*\big(\pi_1(\tilde E,\tilde e_0)\bigr)$, then there exists a continuous map $\theta : E\to \tilde E$ such that $\rho= \tilde \rho\circ\theta$.
        
        Suppose that $\tilde E = \varprojlim \tilde E_i$, where $\tilde \rho_i: \tilde E_i\to X$ is a covering map, for each $i$.  Since $X$ is a Peano continuum, we can find a polyhedral expansion $(P_i)$ of  $X$ such that the projections $\sigma_i: X\to P_i$ induce surjective homomorphisms of fundamental groups.  By Lemma 3.5 in \cite{ConnerHerfortKentPavesic2021}, for each $i$, there exists an $n_i$ and covering map $\tilde q_i: \tilde P_i \to P_{n_i}$ such that $\tilde \rho_i:\tilde E_i\to X$ is homeomorphic to the pullback of $\tilde q_i$ along $\sigma_{n_i}$.

        By Corollary 3.12 of \cite{ConnerHerfortKentPavesic2025}, for each $i$, there exists a continuous map $\theta_i': E \to\tilde{ \tilde P}_i$, where $\eta_i: \tilde{ \tilde P}_i\to P_{n_i}$ is a covering map with $(\sigma_{n_i}\circ \rho)_*\big(\pi_1(E, e_0)\big)= \eta_{i*}\big(\pi_1(\tilde{\tilde P}_i, \tilde{\tilde p}_i)\big)$, for appropriate basepoints.  Since $\im( \rho_{*})\subset\im( \tilde \rho_{*})\subset\im( \tilde \rho_{i*})$, there is a covering map from $\tilde{\tilde P}_i$ to $\tilde P_i$ that makes the diagram below on the right, without $\theta_i$, commute.

\[\begin{tikzcd}
	E & {\tilde{\tilde P}_i} &&& E && {\tilde E_{i+1} } \\
	& {\tilde E_i } & {\tilde P_i} &&&& {\tilde E_i } \\
	& X & {P_{n_i}} &&& X
	\arrow["{\theta_i'}", from=1-1, to=1-2]
	\arrow["{\theta_i}"{pos=0.7}, dashed, from=1-1, to=2-2]
	\arrow["\rho"', from=1-1, to=3-2]
	\arrow[from=1-2, to=2-3]
	\arrow["{\theta_{i+1}}", from=1-5, to=1-7]
	\arrow["{\theta_i}"', from=1-5, to=2-7]
	\arrow["\rho"', from=1-5, to=3-6]
	\arrow[from=1-7, to=2-7]
	\arrow["{\tilde\rho_{i+1}}", shift left=2, curve={height=-40pt}, from=1-7, to=3-6]
	\arrow[from=2-2, to=2-3]
	\arrow["{\tilde\rho_i}", from=2-2, to=3-2]
	\arrow["{\tilde q_i}"', from=2-3, to=3-3]
	\arrow["{\tilde \rho_i}", from=2-7, to=3-6]
	\arrow["{\sigma_{n_i}}"', from=3-2, to=3-3]
\end{tikzcd}\]

Then by the universal properties of pullbacks, there exists a continuous map $\theta_i: E \to \tilde E_i$  such that the full diagram on the left commutes.  Since covering spaces have the unique path-lifting property, $\rho = \tilde \rho_i \circ \theta_i$, and $\rho = \tilde \rho_{i+1} \circ \theta_{i+1}$, it is an exercise to check that the diagram on the right commutes.  Thus the sequence of maps $\theta_i$ induce a map $\theta: E\to \tilde E$ such that $\rho = \tilde \rho\circ \theta$.

    \end{proof}

\begin{thm}\label{thm:minimal Cech}
     Suppose that $\rho:(E, e_0)\to (X,x_0)$ is the unique path-connected pro-covering over a Peano continuum $X$ such that $\rho_*\bigl(\pi_1(E,e_0)\bigr)$ is the $\mathbb Z$-kernel of $\pi_1(X,x_0)$.  Then $E$ is the minimal path-connected pro-covering over $X$ which induces the trivial homomorphism on \v Cech cohomology.
\end{thm}

By minimal, we mean that if $\eta: F\to X$ is a path-connected pro-covering that induces the trivial homomorphism on \v Cech cohomology, then there exists a continuous map $\theta: F\to E$ such that $\eta = \rho\circ \theta$.

\begin{proof}
    Fix $\{P_i, \sigma_{i,j}\}$ a polyhedral expansion for $X$.  We will identify the \v Cech cohomology with $[X, \mathbb S^1]$, i.e., homotopy classes of maps from $X$ to $\mathbb S^1$.  Since every map from $X$ to $S^1$ factors through a projection $\sigma_i: X \to P_i$ and $\pi_1(E, e_0)\leq \ker( \sigma_{i*})$, the pro-covering $\rho$ must induce the trivial map from $[X, \mathbb S^1]$ to $[E, \mathbb S^1]$.

    Let $\eta: F\to X$ be a pro-covering that induces the trivial homomorphism on \v Cech cohomology.  Then $\eta_*\bigl(\pi_1(F,f_0)\bigr)$ must be contained in the $\Z$-kernel of $\pi_1(X,x_0)$.    Thus  $\eta_*\bigl(\pi_1(F,f_0)\bigr)\subset \rho_{*} \bigl(\pi_1(E, e_0)\bigr)$ and then, by Lemma \ref{lem: continuous lift}, there exists a continuous map $\theta :F\to E$ such that $\eta = \rho\circ \theta$.
\end{proof}

\section{Cotorsion subgroups and splittings}\label{cotorsion}

We will now turn our attention to proving Theorem \ref{thm: higman complete}.   Thus in many settings this allows one to split the first homology group into the product of a Baer-Specker group and a cotorsion group.

\begin{defn}\label{defn: cotorsion}
    A group $G$ is \emph{cotorsion} if $\operatorname{Ext}(\mathbb Q, G)=0$, which is equivalent to every exact sequence $0\to G\to H\to K\to 0$ is split if $K$ is torsion-free, see \cite[Chapter 9]{Fuchs2015}.
\end{defn}

\begin{defn}
    A group $G$ is \emph{Higman complete} if every system of equations $x_i = w_i(f_i,x_{i+1})$, where $w_i$ is a word with variable $x_{i+1}$ and constants $f_i$, has a solution for any sequence of elements $(f_i)$ from $G$.
\end{defn}

\begin{lem}
    An abelian group $G$ is Higman complete if and only if every system of equations of the form $y_i=(g_iy_{i+1})^{n_i}$ for given $g_i\in G$ and integers $n_i$ admits a solution sequence $(y_i)_{i\ge1}$.
\end{lem}

\begin{proof}
    To prove that $G$ is Higman complete if every system of equations $y_i=(g_iy_{i+1})^{n_i}$ admits a solution sequence, fix an arbitrary system of equations $x_i=w_i(f_i,x_{i+1})$ where $f_i\in G$ and the $w_i$ are bivariate polynomials. Since $G$ is abelian, any such polynomial can be taken in the form \[w_i(f_i,x_{i+1})=f_ix_{i+1}^{n_i}.\]

    Therefore the equations in the system are of the form $x_i=f_ix_{i+1}^{n_i}$. Setting $g_{i-1}:=f_i$ and introducing new variables $y_i:=f_i^{-1}x_i$ transforms the system into $y_i=(g_iy_{i+1})^{n_i}$. Since the latter, by assumption, admits a solution sequence $(a_i)_{i\ge1}$, then $w_i(f_i,x_{i+1})$ admits the solution sequence $(f_ia_i)_{i\ge1}$. Hence $G$ is Higman complete.

\medskip
    If $G$ is Higman complete, then letting $x_i=y_i$ and $f_i:=g_i^{n_i}$ the  system $y_i = (g_iy_{i+1})^{n_i}$ takes the form $x_i=f_ix_{i+1}^{n_i}$ and thus must have a solution sequence $(a_i)_{i\ge1}$ which is also solution sequence to the  system $y_i = (g_iy_{i+1})^{n_i}$.

\end{proof}

\begin{rmk}
    Herfort and Hojka showed that an abelian group is Higman complete if and only if it is cotorsion.  They also showed that the fundamental group of a space where every based homotopy class of loops has arbitrarily small representatives, such spaces are called small loop spaces, is Higman complete \cite[Theorems 3 \& 4]{HerfortHojka17}.  They proved that small loop spaces are Higman complete by finding small enough based representatives for the coefficients $f_i$.

    Here we will be considering equations with the coefficients coming from quotients of subgroups of the fundamental groups where every element has a representative which is a finite product of conjugates of arbitrarily small loops based at potentially different points, i.e., every coefficient has a representative in every Spanier subgroup of the fundamental group.  The difficulty is that one must keep track of a larger and larger set of base points at which the coefficients have small representatives, which is where Section \ref{construction} is needed.
\end{rmk}

\begin{lem}\label{lem: spanier small coset representatives}
    Let $X$ be a Peano continuum and endow $\pi_1(X,x_0)$ with the shape topology. Fix $H$ a subgroup of $\pi_1(X,x_0)$ and $Z$ a subset of the closure of $H$ in $\pi_1(X,x_0)$.  Then for any open cover $\mathcal U$ and any $z\in Z$, there exists an element $g\in \pi_1(X, x_0; \U)$ such that  $zH =gH$.
\end{lem}

\begin{proof}
    Let $\{P_i, \sigma_{i,j}\}$  be a polyhedral expansion of $X$ with projection maps $\sigma_i: X \to X_i$.    By Lemma \ref{lem: factor}, there exists an $i$ such that $\ker(\sigma_{i*}) \subset \pi_1(X, x_0; \U)$. Since $Z$ is contained in the closure of $H$, we have that $\sigma_{i*}(Z) \subset \sigma_{i*}(H)$.  Thus $Z\cdot\ker(\sigma_{i*}) \subset  H \cdot\ker(\sigma_{i*})=  \ker(\sigma_{i*})\cdot H$. Hence $Z \subset \ker(\sigma_{i*})\cdot H\subset \pi_1(X, x_0; \U)\cdot H$.
\end{proof}

\begin{thm}\label{thm: higman complete}
    Let $X$ be a Peano continuum and $H$ a subgroup of $\pi_1(X,x_0)$ containing the commutator subgroup.  Then $\overline H/ H$ is cotorsion, where $\overline H$ is the closure of $H$ in the shape topology.
\end{thm}

\begin{proof}
    Let $\gamma$ be a nullhomotopic space filling curve.  Let $l_i = \max\limits_{1\leq j\leq 2^i} \diam\bigl(\im \bigl(\gamma|_{[\frac{j-1}{2^i}, \frac{j}{2^i}]}\bigr) \bigr)$.  Fix a system of equations $y_i = \bigl(g_i y_{i+1}\bigr)^{n_{i}}$ with $g_i \in \overline H/H $.

    By Lemma \ref{lem: spanier small coset representatives}, each coset $g_i$ can be represented by an element of the Spanier subgroup of a cover with mesh at most $l_i$. Hence, Lemma \ref{lem: right basepoints}, allows us to find loops ${}_ib_\bfs$ based at $\gamma(\bfs)$, for $|\bfs|\leq i$, with diameter at most $4l_i$ such that $g_i = [h_i]\cdot H $ where $h_i =_C \stackrel[{|\bfs|\leq i }]{}{\ast} (\iota_\bfs*{}_ib_\bfs*\overline\iota_\bfs)$

    By Lemma \ref{higman solutions}, there exists loops $\beta_i$ such that $[\beta_i]\in \overline H$ and
    \[\beta_i =_C\left(\stackrel[{|\bfs|\leq i}]{}{\ast}(\eta_{\bfs}*\bfsi*\overline{\eta_{\bfs}}) \right)^{n_i}*(\beta_{i+1})^{n_i}\]
    which implies that $[\beta_i]\cdot H = \Bigl(g_i \bigl([\beta_{i+1}]\cdot H\bigr)\Bigr)^{n_i}$.  In other words, $[\beta_i]\cdot H$ is a solution in $\overline H/H $ to the system of equations $y_i = \bigl(g_i y_{i+1}\bigr)^{n_{i}}$.
\end{proof}

  Recall the \v{C}ech homology of $X$ is   $\check H_1(X)= \varprojlim H_1(P_i)$ where $\{P_i, \sigma_{i,j}\}$ is a polyhedral expansion of $X$.

\begin{thm}\label{thm:short exact fundamental group}
  Suppose that $X$ is a Peano continuum.  If $C$ is the commutator subgroup of $\pi_1(X,x_0)$, then we have the following short exact sequence of groups,
  \[0\to \overline C \to \pi_1(X,x_0) \to \check H_1(X) \to 0,\]  where $\overline C$ is the closure of $C$ in the shape topology.
\end{thm}

\begin{proof}
    Let $\{P_i,\sigma_{i,j}\}$ be a polyhedral expansion of $X$.  Let $E_i$ be the covering space of $X$ corresponding to the subgroup $C\cdot \ker(\sigma_{i*})$ and $E$ be the corresponding pro-covering.  Then $\pi_1(E,e_0) = \cap_i C\cdot\ker(\sigma_{i*})$.  Since $\ker(\sigma_{i*})$ is a neighborhood basis of the identity,  $\cap_i C\cdot\ker(\sigma_{i*})=\overline C$.  Let $\tilde P_i$ be the covering space corresponding to the commutator subgroup of $\pi_1(P_i, p_i)$.   By the discussion preceding Theorem 3.7 in \cite{ConnerHerfortKentPavesic2021}, we have the following exact sequence of groups
    \[0\to \pi_1(E)\to \pi_1(X,x_0) \to \varprojlim\pi_1(P_i)/\pi_1(\tilde P_i) \to \pi_0(E) \to 0,\]

    \noindent where following the conventions of \cite{ConnerHerfortKentPavesic2021} we have omitted  basepoints and identified the fundamental group of a covering space with its image in the fundamental group of the base space.

    Since $\pi_1(\tilde P_i, \tilde p_i)$ is the commutator subgroup of $\pi_1(P_i, p_i)$, we have $\pi_1(P_i)/\pi_1(\tilde P_i) = H_1(P_i, \Z)$.  Hence, the inverse limit term is the  Cech homology of $X$.  Since $E$ is path-connected, $\pi_0(E)$ is trivial and $\pi_1(X)$ maps surjectively onto $\check H_1(X,\Z)$.  This completes the proof.
\end{proof}

Abelianizing this short exact sequence gives us the following restatement of Theorem \ref{thm:short exact fundamental group} in terms of the first homology.

\begin{thm}\label{thm:short exact homology}
  Suppose that $X$ is a Peano continuum and $C$ is the commutator subgroup of $\pi_1(X,x_0)$.  Then we have the following short exact sequence of groups,
  \[0\to \overline C/C \to H_1(X,\Z) \to \check H_1(X) \to 0\]
  and $\overline C/C$ is a cotorsion group.  
\end{thm}

If $\check H_1(X)$ is torsion-free, then the sequence splits since the kernel is cotorsion by Theorem \ref{thm: higman complete}.

\begin{thm}\label{thmm: splitting}
  If $X$ is a Peano continuum  with $\check H_1(X)$ torsion-free, then the first singular homology of $X$ splits as $H_1(X) = \check H_1(X)\oplus K$, where $K$ is the first homology shape kernel of $X$, a cotorsion group.
\end{thm}

Note, if $X$ has a polyhedral expansion $\{P_i\}$ such that $H_1(P_i)$ is torsion-free, then $\check H_1(X)$ is torsion-free.

\begin{cor}\label{cor: homology splitting}
    If $X\subset \mathbb R^3$ is a Peano continuum, then $H_1(X) = \mathbb Z^\lambda\oplus K$, where $K$ is the first homology shape kernel of $X$, a cotorsion group, and $\lambda$ is a countable cardinal.

\end{cor}

\begin{proof}
    Let $P_i$ be the closed $(1/i)$-neighborhood of $X$ in $\mathbb R^3$.  Using Alexander duality and the fact that the first Cech cohomology group is torsion-free, we can see that $H_1(P_i, \mathbb Z)$ is torsion free.  Hence $H_1(P_i, \mathbb Z)$ is a finitely generated, free abelian group.  If the rank of $H_1(P_i)$ is bounded, then $H_1(X,\Z)= \Z^m$ for some $m\in\N$.  If the rank of $H_1(P_i)$ is unbounded, then it is an exercise to see that $\check H_1(X, \mathbb Z)= \Z^\N$.
\end{proof}


\section{A criterion for non path-connected pro-coverings}

\begin{lem}\label{lem: not path-connected}
    Let $X$ be a Peano continuum and suppose that $E = \varprojlim\{E_i, \nu_{i,j}\}$ is a pro-covering of $X$.  If there exists a continuous map $f:X \to P$ to a countable CW-complex $P$ such that $\rho_{i*}\bigl(\pi_1(E_i, e_i)\bigr)\cdot \ker(f_*)$ is not eventually constant, then $E$ is not path-connected.
\end{lem}

\begin{proof}
    Let $H_i = \rho_{i*}\bigl(\pi_1(E_i, e_i)\bigr)$,  $K_i = f_*\circ \rho_{i*}\bigl(\pi_1(E_i, e_i)\bigr)$, and $\theta_i: \tilde P_i \to P$ be the covering space of $P$ with $\theta_{i*}\bigl(\pi_1(\tilde P_i,p_i)\bigr) = K_i$.  Notice that if $\rho_{i*}\bigl(\pi_1(E_i, e_i)\bigr)\cdot \ker(f_*)$ is not eventually constant, then $K_i$ is not eventually constant.  Thus Corollary 2.3 of \cite{ConnerHerfortKentPavesic2021} implies that $\varprojlim \tilde P_i$ is not path-connected. Since $\varprojlim \tilde P_i$ is not path-connected, there exists a nested sequence of cosets $K_ig_i$ such that $\bigcap\limits_i K_ig_i= \emptyset$.  Thus $K_i(g_ig_1^{-1})$ is also a nested sequence of cosets such that $\bigcap\limits_i K_i(g_ig_1^{-1})= \emptyset$.  Notice that $g_ig_{i-1}^{-1} \in K_{i-1}$ for $i\geq 2$.  Choose $h_i \in H_{i-1}$ such that $f_*(h_i) = g_ig_{i-1}^{-1}$.  This gives a nested sequence of cosets $H_i(h_i\cdots h_2)$ in $\pi_1(X,x_0)$ such that $f_*\bigl(H_i(h_i\cdots h_2)\bigr) = K_{i}(g_ig_1^{-1})$.  Since $\bigcap\limits_i K_{i}(g_ig_1^{-1})=\emptyset$, we have $\bigcap\limits_i H_i(h_i\cdots h_2) = \emptyset$ and $E$ is not path-connected.
\end{proof}

\begin{cor}\label{cor: not path-connected}
Let $X$ be a Peano continuum and suppose that $E = \varprojlim\{E_i, \nu_{i,j}\}$ is a pro-covering of $X$.  If there exists an open subgroup $K$ of $\pi_1(X,x_0)$ such that $\rho_{i*}\bigl(\pi_1(E_i, e_i)\bigr)\cdot K$ is not eventually constant, then $E$ is not path-connected.
\end{cor}

\begin{proof}
  By Lemma \ref{lem: factor}, there exists a continuous map $f: X\to P$ such that $\ker(f_*)\leq K$, where $P$ is a term in some polyhedral expansion of $X$.  Since $\rho_{i*}\bigl(\pi_1(E_i, e_i)\bigr)\cdot K$ is not eventually constant, $\rho_{i*}\bigl(\pi_1(E_i, e_i)\bigr)\cdot \ker(f_*)$ is not eventually constant.  The result then follows from Lemma \ref{lem: not path-connected}.
\end{proof}

\begin{thm}\label{thm: characterization in terms of maps}
  Let $X$ be a Peano continuum and $E = \varprojlim\{E_i, \nu_{i,j}\}$ be a pro-covering of $X$ such that $\rho_{i*}\bigl(\pi_1(E_i,e_i)\bigr)$ contains the commutator subgroup of $\pi_1(X,x_0)$.  Then $E$ is path-connected if and only if the sequence of subgroups $\rho_{i*}\bigl(\pi_1(E_i,e_i)\bigr)\cdot \ker{(f_*)}$ is eventually constant for every continuous map $f:X \to P$ to a finite polyhedral complex.
\end{thm}

\begin{proof}
    Corollary \ref{cor: not path-connected} proves that if $E$ is path-connected, then $\rho_{i*}\bigl(\pi_1(E_i,e_i)\bigr)\cdot \ker(f_*)$ is eventually constant for every continuous map $f:X \to P$ to a polyhedral complex.

    Suppose that for every map $f:X \to P$ to a polyhedral complex the sequence of subgroups $\rho_{i*}\bigl(\pi_1(E_i,e_i)\bigr)\cdot \ker{(f_*)}$ is eventually constant.  Let $\{P_i, \sigma_{i,j}\}$ be a polyhedral expansion of $X$.

    Let $H_i=  \rho_{i*}\bigl(\pi_1(E_i,e_i)\bigr)$.  By hypothesis, for each $k$, there exists an $N(k)$ such that $H_i\cdot\ker(\sigma_k) = H_j\cdot\ker(\sigma_k)$ for all $i,j\geq N(k)$.  We may assume that $k<N(k)< N(k+1)$.  This defines a strictly increasing function $N:\mathbb N\to \mathbb N$.

    To show that $E$ is path-connected, we need to show that for every nested sequence of cosets of $H_i$, their intersection is nonempty.  Let $H_ig_i$ be a nested sequence of cosets, which implies that $g_{j}g_i^{-1} \in H_i$ for all $j\geq i$.  We will consider the subsequence given by $\bigl(N^i(1)\bigr)$. Since this subsequence is cofinal, $\bigcap H_{N^i(1)}g_{N^i(1)} = \bigcap H_{i}g_{i}$ and we need only show that the cofinal subsequence of cosets has a non-empty intersection.  After passing to this subsequence, we have that $H_i\cdot\ker(\sigma_k) = H_j\cdot\ker(\sigma_k)$ for all $i,j\geq k$.

    Since $H_i$ is a covering subgroup, Lemma \ref{lem: factor} shows that there exists an $m_i$ such that $\ker(\sigma_{m_i})\subset H_i$.

    Since $H_1\cdot\ker(\sigma_1) = H_2\cdot\ker(\sigma_1)$, we have that $H_1\cdot\ker(\sigma_1)g_1 = H_2\cdot\ker(\sigma_1)g_2$.  Thus there exists $k_1\in \ker(\sigma_1)$ and $a_2\in H_2$ such that $g_1 = a_2k_1g_2$.  Then $g_2 = k_1^{-1}a_2^{-1}g_1$, which implies that $H_2g_2 = H_2(k_1^{-1}a_2^{-1}g_1) = H_2(k_1^{-1}g_1)$ (since $H_2$ contains the commutator subgroup and $a_2^{-1}\in H_2$).  By induction suppose that, for $j< i$, we have defined $k_j\in\ker(\sigma_j)$ such that $H_jg_j = H_j(k_{j-1}^{-1}k_{j-2}^{-1}\cdots k_1^{-1}g_1)$.

    Then $H_i\cdot\ker(\sigma_{i-1})(g_i) = H_{i-1}\cdot\ker(\sigma_{i-1})g_{i-1}= H_{i-1}\cdot\ker(\sigma_{i-1})(k_{i-2}^{-1}k_{i-3}^{-1}\cdots k_1^{-1}g_1)$.  Thus there exists $k_{i-1}\in \ker(\sigma_{i-1})$ and $a_i\in H_i$ such that $a_ik_{i-1}g_i = k_{i-2}^{-1}k_{i-3}^{-1}\cdots k_1^{-1}g_1  $, which gives $g_i = k_{i-1}^{-1}a_i^{-1}k_{i-2}^{-1}k_{i-3}^{-1}\cdots k_1^{-1}g_1$.  Again since $H_i$ contains the commutator subgroup and $a_i^{-1}\in H_i$,  $H_ig_i = H_i(k_{i-1}^{-1}a_i^{-1}k_{i-2}^{-1}k_{i-3}^{-1}\cdots k_1^{-1}g_1) = H_i(k_{i-1}^{-1}k_{i-2}^{-1}k_{i-3}^{-1}\cdots k_1^{-1}g_1)$.

    We can now proceed as in the proof of Lemma \ref{complete rel H_n}.  By Lemma \ref{lem: factor} each element $k_i^{-1}$ can be represented by an element of a Spanier subgroup of a cover with mesh at most $l_i$ where $l_i$ converges to 0.   Hence, Lemma \ref{lem: right basepoints}, allows us to find loops ${}_ib_\bfs$ based at $\gamma(\bfs)$, for $|\bfs|\leq i$, with diameter at most $2l_i$ such that $k_i^{-1} = [h_i] $ where $h_i =_C \stackrel[{|\bfs|\leq i }]{}{\ast} (\iota_\bfs*{}_ib_\bfs*\overline\iota_\bfs)$.

    By Lemma \ref{higman solutions}, there exist loops $\beta_i$ such that \[\beta_i =_C\left(\stackrel[{|\bfs|\leq i}]{}{\ast}(\iota_{\bfs}*\bfsi*\overline{\iota_{\bfs}}) \right)*(\beta_{i+1})=_C h_i*\beta_{i+1},\]

    which implies that $\beta_1 =_C h_1*h_2*\cdots *h_i*\beta_{i+1}$ for all $i$.

    Let $\alpha$ be a representative of $g_1$.  Fix an $i\in\mathbb N$.  Then there exists a $j\geq i$ such that $[\beta_j]\in \ker(\sigma_{m_i*})\subset H_i$.  Thus $\beta_1*\alpha =_C h_1*\cdots h_{j-1}*\beta_j*\alpha$.  Since the commutator subgroup in contained in $H_i$, we have that $ H_i[\beta_1*\alpha] = H_i [\beta_j*h_{j-1}*\cdots * h_1*\alpha] = H_i [h_{i-1}*\cdots *h_1*\alpha]= H_i(k_{j-1}^{-1}\cdots k_1^{-1}g_1)\supset H_j(k_{j-1}^{-1}\cdots k_1^{-1}g_1)= H_jg_j$.  Since both, $H_i[\beta_1*\alpha]$ and $H_ig_i$, contain $H_jg_j$, they are the same coset and $[\beta_1*\alpha]\in H_ig_i$.  Thus $[\beta_1*\alpha]\in H_{i}g_{i}$ for all $i$, as desired.

\end{proof}

\begin{cor}
  Let $X$ be a Peano continuum and $E = \varprojlim\{E_i, \nu_{i,j}\}$ be a pro-covering of $X$ such that $\rho_{i*}\bigl(\pi_1(E_i,e_i)\bigr)$ contains the commutator subgroup of $\pi_1(X,x_0)$.  Then the following are equivalent.

  \begin{enumerate}
      \item $E$ is path-connected.
      \item For every open subgroup $U$ of $\pi_1(X,x_0)$, the sequence $\rho_{i*}\bigl(\pi_1(E_i,e_i)\bigr)\cdot U$ is eventually constant.
      \item For every $n$, the sequence $\rho_{i*}\bigl(\pi_1(E_i,e_i)\bigr)\cdot U_n$ is eventually constant where $\{U_n\}$ is a local basis of the identity element in $\pi_1(X,x_0)$.
  \end{enumerate}

\end{cor}


\bibliographystyle{plain}
\bibliography{bib}

\end{document}